\documentclass[a4paper]{amsart}
\usepackage{latexsym,bm}
\usepackage{amsmath,amsthm,amsfonts,amssymb,mathrsfs}

\usepackage{times,amssymb,amsmath,mathrsfs}

\let\<=\langle
\let\>=\rangle
\def\({\big(}
\def\){\big)}

\def\pmod#1{\text{ }(\text{mod } #1)\,}

\let\Sym=\BS

\newcommand{\N}{\mathbb N}
\newcommand{\Z}{\mathbb Z}

\newcommand{\HH}{\mathscr{H}}

\DeclareMathOperator\id{id}
\DeclareMathOperator\Id{id}

\newcommand{\Q}{\mathbb Q}

\def\A{\mathcal{A}}

\newcommand{\lam}{\lambda}

\newcommand{\eps}{\epsilon}

\DeclareMathOperator{\diag}{Diag}
\DeclareMathOperator{\rank}{rank}

\DeclareMathOperator{\Ker}{Ker}

\newcounter{main}
\theoremstyle{plain}

\numberwithin{equation}{section}
\newtheorem{prop}[equation]{Proposition}
\newtheorem{thm}[equation]{Theorem}
\newtheorem{cor}[equation]{Corollary}
\newtheorem{lem}[equation]{Lemma}
\newtheorem{conj}[equation]{Conjecture}
\theoremstyle{definition}
\newtheorem{dfn}[equation]{Definition}
\theoremstyle{remark}

  {\refstepcounter{equation}\trivlist
   \item[\hskip\labelsep\theequation.~\textbf{Example#1}\space]
   \ignorespaces
  }{\unskip\nobreak\hfil%
    \penalty50\hskip2em\hbox{}\nobreak\hfil$\Diamond$%
    \parfillskip=0pt\finalhyphendemerits=0\penalty-100\endtrivlist}

{\catcode`\|=\active
  \gdef\set#1{\mathinner{\lbrace\,{\mathcode`\|"8000%
                                   \let|\midvert #1}\,\rbrace}}
}
\def\midvert{\egroup\mid\bgroup}

\begin{document}
\title{Integral $u$-deformed involution modules}

 \author{Jun Hu}
   \address{School of Mathematics and Statistics\\
  Beijing Institute of Technology\\
  Beijing, 100081, P.R. China}
  \email{junhu404@bit.edu.cn}

 \author{Yujiao Sun}
   \address{School of Mathematics and Statistics\\
  Beijing Institute of Technology\\
  Beijing, 100081, P.R. China}
\email{yujiaosun@bit.edu.cn}

\begin{abstract}
Let $(W,S)$ be a Coxeter system and $\ast$ an automorphism of $W$ with order $\leq 2$ and $S^{\ast}=S$.
Lusztig and Vogan (\cite{Lu1}, \cite{LV1}) have introduced a $u$-deformed version $M_u$ of Kottwitz's involution module over the Iwahori--Hecke algebra $\HH_{u}(W)$ with Hecke parameter $u^2$, where $u$ is an indeterminate. Lusztig has proved that $M_u$ is isomorphic to the left $\HH_{u}(W)$-submodule of $\hat{\HH}_u$ generated by $X_{\emptyset}:=\sum_{w^*=w\in W}u^{-\ell(w)}T_w$, where $\hat{\HH}_u$ is the vector space consisting of all formal (possibly infinite) sums $\sum_{x\in W}c_xT_x$ ($c_x\in\Q(u)$ for each $x$). He speculated that one can extend this by replacing $u$ with any $\lam\in \mathbb{C}\setminus\{0,1,-1\}$.  In this paper, we give a positive answer to his speculation for any $\lam\in K\setminus\{0,1,-1\}$ and any $W$, where $K$ is an arbitrary ground field.
\end{abstract}

\maketitle


\section{Introduction}

Let $(W,S)$ be a fixed Coxeter system and $\ast$ be a fixed automorphism of $W$ with order $\leq 2$ and such that $S^\ast=S$. That is, $s^{\ast}\in S$ for any $s\in S$. Let $\ell: W\rightarrow\mathbb{N}$ be the usual length function on $W$. If $w\in W$ then by definition
$$\ell(w):=\min\{k \mid w=s_{i_1}\dots s_{i_k} \text{ for some }s_{i_1},\dots,s_{i_k}\in S\}.$$

\begin{dfn} \label{twistedinvolutions} We define $\mathbf{I}_{\ast}:=\bigl\{w\in W\bigm|w^{\ast}=w^{-1}\bigr\}$.
The elements of $I_{\ast}$ will be called {\it twisted involutions} relative to $\ast$.
\end{dfn}

Let $u$ be an indeterminate over $\Q$ (the field of rational numbers).

\begin{dfn}[\cite{GP, KL}] \label{IH} Let $\HH_{u}:=\HH_{u}(W)$ be the associative unital $\Q(u)$-algebra with a $\Q(u)$-basis $\{T_w \mid w\in W\}$and multiplication defined by $$\begin{aligned}
& T_w T_{w'}=T_{ww'} \,\,\,\text{if $\ell(ww')=\ell(w)+\ell(w')$;}\\
& (T_s+1)(T_s-u^2)=0 \,\,\,\text{if $s\in S$.}
\end{aligned}
$$
We call $\HH_{u}(W)$ the Iwahori--Hecke algebra over $\Q(u)$ associated to $(W,S)$ with Hecke parameter $u^2$.
\end{dfn}

Let $\mathcal{A}:=\Z[u,u^{-1}]$ be the ring of Laurent polynomials on $u$. Let $\HH_{\A,u}$ be the $\A$-subalgebra of $\HH_{u}$ generated by $\{T_w \mid w\in W\}$. Then $\HH_{\A,u}$ is a natural $\A$-form of $\HH_{u}$ and isomorphic to the abstract $\A$-algebra defined by the same generators and relations as in Definition \ref{IH}. For any field $K$ and any $\lam\in K^\times$, there is a unique ring homomorphism $\phi: \A\rightarrow K$ satisfying that $\phi(u)=\lam$. We define
$\HH_\lam:=K\otimes_{\A}\HH_u$ and call $\HH_\lam$ the specialized Iwahori--Hecke algebra associated to $(W,S)$ with Hecke parameter $\lam^2$.

Let $M_u$ be a $\Q(u)$-linear space with a $\Q(u)$-basis $\{a_z \mid z\in \mathbf{I}_\ast\}$.

\begin{lem}[\cite{Lu1, LV1}] \label{LVInvoModu} There is a unique $\HH_{u}$-module structure on $M_u$ such that for any $s\in S$ and any $w\in \mathbf{I}_\ast$, $$\begin{aligned}
& T_s a_w=ua_w+(u+1)a_{sw}\quad\text{if\, $sw=ws^\ast>w$;}\\
& T_s a_w=(u^2-u-1)a_w+(u^2-u)a_{sw}\quad\text{if\, $sw=ws^\ast<w$;}\\
& T_s a_w=a_{sws^\ast}\quad\text{if\, $sw\neq ws^\ast>w$;}\\
& T_s a_w=(u^2-1)a_w+u^2a_{sws^\ast}\quad\text{if\, $sw\neq ws^\ast<w$.}\\
\end{aligned}
$$
\end{lem}

When $u$ is specialized to $1$, the module $M_u$ degenerates to the involution module introduced more than fifteen years ago by Kottwitz \cite{Ko}. Kottwitz
found the module by analyzing Langlands¡¯ theory of stable characters for real groups. He gave a conjectural description of it (later established by Casselman) in terms of the Kazhdan-Lusztig left cell representations of the Weyl group $W$. One interesting fact about the module $M_u$ is that if $W$ is of finite classical type then any irreducible representation $V$ appears as a component of $M_u$ if and only if $V$ is a special irreducible representation of $W$ in the sense of \cite{Lu0}. For this reason, we call $M_u$ the $u$-deformed involution modules.

In a series of papers \cite{Lu1}, \cite{Lu2}, \cite{Lu3}, \cite{LV1}, Lusztig and Vogan have studied the $u$-deformed involution modules systematically. A bar invariant canonical basis for $M_u$ and certain coefficient polynomials $P_{y,w}^{\sigma}$ were introduced, which can be regarded as some twisted analogue of  the classical well-known Kazhdan--Lusztig basis and Kazhdan--Lusztig polynomials (\cite{KL}).

Let $\hat{\HH}_{\mathcal{A},u}$ (resp., $\hat{\HH}_u$) be the free $\mathcal{A}$-module (resp., the $\Q(u)$-vector space) consisting of all formal (possibly infinite) sums $\sum_{x\in W}c_xT_x$, where $c_x\in\mathcal{A}$ (resp., $c_x\in\Q(u)$) for each $x\in W$.

\begin{dfn} \text{(\cite{Lu2})} We define $$
X_{\emptyset}:=\sum_{x\in W, x^\ast=x}u^{-\ell(x)}T_x \in \hat{\HH}_{\mathcal{A},u} \subseteq \hat{\HH}_u.
$$
\end{dfn}

\begin{thm}[\cite{Lu3}] \label{iso1}  The map $\mu: M_u\rightarrow \hat{\HH}_u$ which sends $a_1$ to $X_\emptyset$ can be extended uniquely to a left $\HH_u$-module isomorphism $M_u\cong\HH_u X_\emptyset$.
\end{thm}
Note that the above theorem was proved in \cite{HZH1} by the first author and Jing Zhang in the special case when $W=\Sym_n$ and $\ast=\Id$.

\begin{dfn} We define $$
\A_{\pm 1}:=\Z[u, u^{-1}, (u+1)^{-1}, (u-1)^{-1}] .
$$
\end{dfn}

Let $M_{\A,u}$ be the free $\A$-submodule of $M_u$ generated by $\{a_z\mid z\in \mathbf{I}_\ast\}$. By Lemma \ref{LVInvoModu}, it is clear that $M_{\A,u}$ naturally becomes a left $\HH_{\A,u}$-module. We set $$
M_{\A_{\pm 1},u}:=\A_{\pm 1}\otimes_{\A}M_{\A,u},
\quad
\HH_{\A_{\pm 1},u}:=\A_{\pm 1}\otimes_{\A}\HH_{\A,u},
\quad
\hat{\HH}_{\A_{\pm 1},u}:=\A_{\pm 1}\otimes_{\A}\hat{\HH}_{\A,u}.
$$
For any ring homomorphism $\phi: \A_{\pm 1}\rightarrow K$ with $\lam=\phi(u)\in K\setminus\{0,1,-1\}$, we define
\begin{align*}
M_{\lam}:=& K\otimes_{\A_{\pm 1}}M_{\A_{\pm 1},u},\\
\HH_\lam:=& K\otimes_{\A_{\pm 1}}\HH_{\A_{\pm 1},u}\cong K\otimes_{\A}\HH_{\A,u},\\
\hat{\HH}_\lam:=& K\otimes_{\A_{\pm 1}}\hat{\HH}_{\A_{\pm 1},u}\cong K\otimes_{\A}\hat{\HH}_{\A,u}.
\end{align*}
If $W$ is finite, then $\HH_{\A_{\pm 1},u}=\hat{\HH}_{\A_{\pm 1},u}$, $\HH_u=\hat{\HH}_u$ and $\HH_\lam=\hat{\HH}_\lam$.
Note that $\HH_{\A,u}$ (resp., $\HH_{\A_{\pm 1},u}$) is a free $\A$-module (resp., $\A_{\pm 1}$-module) with basis $\{T_w\mid w\in W\}$. For simplicity, we shall often abbreviate $1_K\otimes_{\A}T_w$ and $1_K\otimes_{\A_{\pm 1}}T_w$  as $T_w$.

By some calculations in small ranks, Lusztig has speculated in \cite[\S4.10]{Lu3} that Theorem \ref{iso1} might be extended to the setting of specialized version $\hat{\HH}_\lam$ of $\hat{\HH}_u$ for arbitrary $\lam\in\mathbb{C}\setminus\{0,1,-1\}$ when $W$ is finite. Therefore, it is natural to make the following conjecture.

\begin{conj} \label{conj1} Let $K$ be a field and $\lam\in K\setminus\{0,1,-1\}$. Let $(W,S)$ be an arbitrary Coxeter system.

(1) The map $\mu$ restricts to a left $\HH_{\A_{\pm 1},u}$-module isomorphism $M_{\A_{\pm 1},u}\cong\HH_{\A_{\pm 1},u}X_{\emptyset}$.

(2) For any ring homomorphism $\phi: \A_{\pm 1}\rightarrow K$ with $\lam=\phi(u)$, the map which sends $1_K\otimes_{\A_{\pm 1}}a_1$ to $1_K\otimes_{\A_{\pm 1}}X_{\emptyset}$ can be extended uniquely to a well-defined left $\HH_\lam$-module isomorphism $M_\lam\cong\HH_\lam(1_K\otimes_{\A_{\pm 1}} X_\emptyset)$.

(3) For any ring homomorphism $\phi: \A_{\pm 1}\rightarrow K$ with $\lam=\phi(u)$, the canonical map
$$\begin{aligned}
\iota_K:\,\,K\otimes_{\A_{\pm 1}}\HH_{\A_{\pm 1},u} X_\emptyset& \rightarrow\HH_\lam (1_K\otimes_{\A_{\pm 1}}X_{\emptyset})\\
r\otimes_{\A_{\pm 1}} hX_\emptyset &\mapsto (r\otimes_{\A_{\pm 1}} h)(1_K\otimes_{\A_{\pm 1}}X_{\emptyset}),\quad\,\forall\, r\in K, h\in\HH_{\A_{\pm 1},u} ,
\end{aligned}
$$
is a left $\HH_\lam$-module isomorphism.

(4) If $W$ is finite, then $\HH_{\A_{\pm 1},u}X_{\emptyset}$ is a pure and free $\A_{\pm 1}$-submodule of ${\HH}_{\A_{\pm 1},u}$.

\end{conj}

The purpose of this paper is to give a proof of the above conjecture and thus give a positive answer to Lusztig's speculation. As an application of our main result, we obtain a new integral basis for the module $M_u$ and for the module $\HH_{\A_{\pm 1},u} X_\emptyset$, see Corollary \ref{Aneg1basis} and Corollary \ref{newBasis2}.

\bigskip

\section{Proof of Conjecture \ref{conj1}}

The purpose of this section is to give a proof of Conjecture \ref{conj1}.

\begin{dfn}[\cite{Hu2}] \label{twistedinvolutions2} For any $w\in \mathbf{I}_\ast$ and $s\in S$, we define $$
s\ltimes w:=\begin{cases} sw &\text{if $sw=ws^\ast$;}\\
sws^\ast &\text{if $sw\neq ws^\ast$.}
\end{cases}
$$
For any $w\in \mathbf{I}_\ast$ and $s_{i_1},\cdots,s_{i_k}\in S$, we define $$
s_{i_1}\ltimes s_{i_2}\ltimes\cdots\ltimes s_{i_k}\ltimes w:=s_{i_1}\ltimes\bigl(s_{i_2}\ltimes\cdots\ltimes (s_{i_k}\ltimes w)\cdots\bigr) .
$$
\end{dfn}
It is clear that $s\ltimes w\in \mathbf{I}_\ast$ whenever $w\in \mathbf{I}_\ast$ and $s\in S$. Furthermore, $\ltimes$ is in general not associative.

\begin{dfn}[\cite{HMP2, Hu2}] \label{Ireduced} Let $w\in \mathbf{I}_\ast$. If $w=s_{i_1}\ltimes s_{i_2}\ltimes\cdots\ltimes s_{i_k}\ltimes 1$, where $k\in\N$, $s_{i_j}\in S$ for each $j$, then $(s_{i_1},\cdots,s_{i_k})$ is called an $\mathbf{I}_\ast$-expression for $w$. Such an $\mathbf{I}_\ast$-expression for $w$ is reduced if its length $k$ is minimal.
\end{dfn}
We regard the empty sequence $()$ as a reduced $\mathbf{I}_\ast$-expression for $w=1$. Let ``$\leq$" be the Bruhat partial ordering on $W$ defined with respect to $S$ (cf. \cite{Hum}). We write $u<w$ if $u\leq w$ and $u\neq w$.
It follows by induction on $\ell(w)$ that every element $w\in \mathbf{I}_\ast$ has a reduced $\mathbf{I}_\ast$-expression.

\begin{lem}[\cite{Hu2, Hu3}] \label{rankfunc} Let $w\in \mathbf{I}_\ast$. Any reduced $\mathbf{I}_\ast$-expression for $w$ has a common length. Let $\rho: \mathbf{I}_\ast\rightarrow\N$ be the map which assigns $w\in \mathbf{I}_\ast$ to this common length. Then $(\mathbf{I}_\ast,\leq)$ is a graded poset with rank function $\rho$. Moreover, if $s\in S$ then $\rho(s\ltimes w)=\rho(w)\pm 1$, and $\rho(s\ltimes w)=\rho(w)-1$ if and only if $\ell(sw)=\ell(w)-1$.
\end{lem}

\begin{cor}[{\cite[Corollary 2.6]{HZH1}}] \label{length0} Let $w\in \mathbf{I}_\ast$ and $s\in S$. Suppose that $sw\neq ws^\ast$. Then $\ell(sw)=\ell(w)+1$ if and only if $\ell(ws^\ast)=\ell(w)+1$, and if and only if $\ell(s\ltimes w)=\ell(w)+2$. The same is true if we replace ``$+$" by ``$-$".
\end{cor}

\begin{lem} \label{truncate} Let $z\in \mathbf{I}_*$. Let $(s_{i_1},\cdots,s_{i_k})$
(where $k\in \mathbb{N}$) be an arbitrary reduced $\mathbf{I}_*$-expression of $z$.
Then there exist $s_{i_{k+1}},s_{i_{k+2}},\dots,s_{i_r}\in S$ and integers ${k}\leq t_k\leq t_{k-1}\leq\dots\leq t_1=r$ such that $r=\ell(z)$, and for each $1\leq a\leq k$,
\begin{align*}
z_a:=s_{i_a}\ltimes s_{i_{a+1}}\ltimes\cdots \ltimes s_{i_k}\ltimes 1=s_{i_a}s_{i_{a+1}}\dots s_{i_k}s_{i_{k+1}}s_{i_{k+2}}\cdots s_{i_{t_a}},
\end{align*}
$\rho(z_a)=k-a+1$, $\ell(z_a)=t_a-a+1$. In particular, $s_{i_a}s_{i_{a+1}}\dots s_{i_{t_a}}$ is a reduced expression of $z_a$, $s_{i_a}\cdots s_{i_k}$ is a reduced expression and $(i_{k+1},i_{k+2},\dots,i_r)$ is uniquely determined by the reduced $\mathbf{I}_*$-expression $(s_{i_1},\cdots,s_{i_k})$.
\end{lem}

\begin{proof} This follows from Lemma \ref{rankfunc} and Corollary \ref{length0} and an induction on $k$.
\end{proof}

\begin{dfn} \label{sigmaZ} For each $z\in \mathbf{I}_\ast$ and each reduced $\mathbf{I}_\ast$-expression $\sigma=(s_{i_1},\cdots,s_{i_k})$ of $z$, we define \begin{equation}
\sigma_z:=s_{i_1}\dots s_{i_k}\in W.
\end{equation}
\end{dfn}
In particular, we have $\sigma_1=1$ and $\rho(z)=\ell(\sigma_z)$ for any $z\in \mathbf{I}_\ast$. In general, $\sigma_z$ depends on the choice of the reduced $\mathbf{I}_\ast$-expression $(s_{i_1},\cdots,s_{i_k})$ of $z$.

\begin{dfn} [{\cite[Proposition 2.5]{Hu3}}, {\cite[Proposition 2.2]{Mar}}] Let $w\in \mathbf{I}_\ast$ and $(s_{i_1},\cdots,s_{i_k})$ be a reduced $\mathbf{I}_*$-expression of $w$. We define $$
w_0:=w,\,\, w_t:=s_{i_t}\ltimes w_{t-1},\,\,\text{for $1\leq t\leq k$.}
$$
Define $\ell^*: \mathbf{I}_\ast\rightarrow\N$ by $$
\ell^*(w):=\#\{1\leq t\leq k\mid s_{i_t}w_t=w_ts_{i_t}^\ast\} .
$$
\end{dfn}
The notation $\ell^*(w)$ we used here was denoted by $\ell^\theta(w)$ in \cite{Hu2} and \cite{Hu3}, and was denoted by $\phi$ in \cite[\S1.5]{Lu3}.
By \cite{Hu3}, $\ell^*(w)$ depends only on $w$ but not on the choice of the reduced $\mathbf{I}_\ast$-expression $(s_{i_1},\cdots,s_{i_k})$ of $w$.

\begin{lem}[{\cite[Theorem 4.8]{Hu1}}, {\cite[\S2.2]{Hu3}}]\label{lem0}  Let $w\in I_\ast$. Then $\rho(w)=(\ell(w)+\ell^*(w))/2$.
\end{lem}

Let $\bar{}: \A\rightarrow\A$ be the ring involution such that $\overline{u^n}=(-u)^{-n}$ for any $n\in\Z$. Let $\epsilon: \mathbf{I}_\ast\rightarrow\{1,-1\}$,
$z\mapsto (-1)^{\rho(z)}$, $\forall\,z\in\mathbf{I}_\ast$. By Lemma \ref{lem0}, our $\epsilon$ coincides with the function $\epsilon$ defined in \cite[\S1.5]{Lu3}.

\begin{dfn}[{\cite[\S1.1]{Lu3}}] \label{Stru1} Let $\{L_z^x \mid z\in \mathbf{I}_\ast, x\in W\}$ be a set of uniquely determined polynomials in $\Z[u]$ such that $$
T_xa_1=\sum_{z\in\mathbf{I}_\ast}L_z^x a_z , \quad\,\forall\, x\in W .
$$
\end{dfn}

\begin{dfn}[{\cite[\S1.6]{Lu3}}] \label{Stru2} For $x\in W, z\in\mathbf{I}_\ast$, we set $$
\widetilde{L}^x_z:=(-1)^{\ell(x)}\eps(z)\overline{L^x_z}.
$$
\end{dfn}

\begin{lem}[{\cite[\S1.7, the 5th line above \S1.8]{Lu3}}] \label{keylem1} For $z\in\mathbf{I}_\ast$, we have $$
\mu(a_z)=\sum_{x\in W}\widetilde{L}_z^x T_x .
$$
\end{lem}

Note that there is a typo in the identity on $\mu(a_z)$ in \cite[\S1.7, the 5th line above \S1.8]{Lu3}). The element $T_z$ in the right hand should be replaced by $T_x$.

\begin{prop}\label{keyprop1} Let $z\in \mathbf{I}_\ast$ and $\sigma=(s_{i_1},\cdots,s_{i_k})$ be a reduced $\mathbf{I}_*$-expression of $z$. Let $(i_{k+1},\cdots,i_r)$ be the unique $(r-k)$-tuple determined by this reduced $\mathbf{I}_*$-expression as described in Lemma \ref{truncate}. Then $z=s_{i_1}\ltimes \cdots \ltimes s_{i_k}\ltimes 1
=s_{i_1}\cdots s_{i_k}s_{i_{k+1}}\cdots s_{i_r}$ with $k=\rho(z)$, $r=\ell(z)$, where $s_{i_1},\dots ,s_{i_r}\in S$, and we have
$L_z^{\sigma_z}=(u+1)^{\ell^*(z)}$, and $L_{w}^{\sigma_z}\neq 0$ only if $\rho(w)<\rho(z)$ and there exists a reduced $\mathbf{I}_*$-expression $\sigma'$ of $w$ such that $\sigma'_{w}<\sigma_{z}$.
Moreover, $L_{w}^{\sigma_z}\in u\mathbb{Z}[u]$ if $w\neq z$.
\end{prop}

\begin{dfn}\label{newOrder} Let $z\in \mathbf{I}_\ast$ and $\sigma=(s_{i_1},\cdots,s_{i_k})$ be a reduced $\mathbf{I}_*$-expression of $z$. We define $$
\mathbf{I}_\ast({\prec_\sigma z}):=\Bigl\{w\in \mathbf{I}_\ast\Bigm|\begin{matrix}\text{$\rho(w)<\rho(z)$ and there exists a reduced $\mathbf{I}_*$-expression $\sigma'$}\\
\text{of $w$ such that $\sigma'_{w}<\sigma_{z}$.}\end{matrix}\Bigr\}
$$
\end{dfn}

Then Proposition \ref{keyprop1} is equivalent to the following identity: \begin{equation}\label{triangularRelation1}
T_{\sigma_z}a_1=(u+1)^{\ell^*(z)}a_{z}+\sum_{w\in \mathbf{I}_\ast({\prec_\sigma z})}{L_{w}^{\sigma_z}a_{w}},
\qquad
L_{w}^{\sigma_z}\in u\mathbb{Z}[u].
\end{equation}

\smallskip
\noindent
{\textbf{Proof of Proposition \ref{keyprop1}}}:  Let $z\in \mathbf{I}_*$. We prove the proposition by induction on $\rho(z)$. If $\rho(z)=0$,
then $z=1$, $\sigma_z=1$ and $T_1a_1=a_1$.

Let $k\in \mathbb{N}^*$. Suppose that the statement holds when $\rho(z)<k$.
Let $z\in \mathbf{I}_*$ with $\rho(z)=k$. We follow the notation and hypothesis in Lemma \ref{truncate} and Definition \ref{sigmaZ}. Then $z=s_{i_1}\ltimes \cdots \ltimes s_{i_k}\ltimes 1
=s_{i_1}\cdots s_{i_k}s_{i_{k+1}}\cdots s_{i_r}$ with $k=\rho(z)$, $r=\ell(z)$ for some $s_{i_1},\dots ,s_{i_r}\in S$.
By definition, $\sigma_z=s_{i_1}\cdots s_{i_k}$.
Let $x'=s_{i_1}\sigma_z=s_{i_2}\cdots s_{i_k}$. Note that $(s_{i_1},\cdots,s_{i_k})$ is a reduced $\mathbf{I}_*$-expression implies that $\sigma':=(s_{i_2},\cdots,s_{i_k})$ is a reduced $\mathbf{I}_\ast$-expression of
$z':=s_{i_2}\ltimes\cdots\ltimes s_{i_k}\ltimes 1$, then $\rho(z')=k-1$ and $x'=\sigma_{z'}$ in the notation of Definition \ref{sigmaZ}.
Now $x'<\sigma_z$ and
\begin{align*}
T_{\sigma_z}a_1=& T_{s_{i_1}}T_{x'}a_1=T_{s_{i_1}}T_{\sigma_{z'}}a_1\\
=& T_{s_{i_1}}
   \left(
(u+1)^{\ell^*(z')}a_{z'}+\sum_{z''\in \mathbf{I}_\ast(\prec_{\sigma'}z')}{L_{z''}^{x'}a_{z''}}
    \right),
\quad
L_{z''}^{x'}\in u\mathbb{Z}[u],\\
=& (u+1)^{\ell^*(z')}T_{s_{i_1}}a_{z'}
+\sum_{z''\in \mathbf{I}_\ast(\prec_{\sigma'}z')}{L_{z''}^{x}T_{s_{i_1}}a_{z''}} .
\end{align*}

We consider the first term in the above identity. There are two possibilities:

\smallskip
\noindent
{\it Case 1.} If $s_{i_1}{z'}={z'}s_{i_1}^*$,
then $z=s_{i_1}\ltimes {z'}=s_{i_1}{z'}$.
Thus $$
T_{s_{i_1}}a_{z'}
= ua_{z'}+(u+1)a_{s_{i_1}z'}= ua_{z'}+(u+1)a_{z},
$$
where $\ell^*(z)=\ell^*(z')+1$ and $z'\in\mathbf{I}_\ast(\prec_{\sigma}z)$, as required.

\smallskip
\noindent
{\it Case 2.} If $s_{i_1}{z'} \neq {z'}s_{i_1}^*$,
then $z=s_{i_1}\ltimes {z'}=s_{i_1}{z'} s_{i_1}^*$.
Thus
$$
T_{s_{i_1}}a_{z'}
=a_{s_{i_1}z's_{i_1}^*}= a_{z}, $$
where $\ell^*(z)=\ell^*(z')$, as required.

Therefore, it remains to consider the term $T_{s_{i_1}}a_{z''}$ for each $z''\in \mathbf{I}_\ast(\prec_{\sigma'}z')$.
We know that $\sigma_z=s_{i_1}x'=s_{i_1}\cdots s_{i_k}$ and $\sigma_{z'}=x'=s_{i_2}\cdots s_{i_k}$ are reduced expressions. Combining our assumption $z''\in \mathbf{I}_\ast(\prec_{\sigma'}z')$ and Lemma \ref{LVInvoModu} together we can deduce  that $T_{s_{i_1}}a_{z''}$ is a $\Z[u]$-linear combination of some $a_w$ with $w\in \mathbf{I}_\ast(\prec_{\sigma}z)$. Therefore, we get that
\begin{align*}
T_{\sigma_z}a_1=(u+1)^{\ell^*(z)}a_{z}+\sum_{w\in \mathbf{I}_\ast(\prec_{\sigma}z)}{L_{w}^{\sigma_z}a_{w}},
\end{align*}
where $L_{w}^{\sigma_z}\in u\mathbb{Z}[u]$ for each $w\in \mathbf{I}_\ast(\prec_{\sigma}z)$. This completes the proof of the proposition.
\smallskip

Note that $w\in \mathbf{I}_\ast(\prec_{\sigma}z)$ implies that $\rho(w)<\rho(z)$.

\begin{cor} \label{TwoReducedcor1} Let $z\in \mathbf{I}_\ast$ and $\sigma, \hat{\sigma}$ be two reduced $\mathbf{I}_*$-expressions of $z$. Then $$\begin{aligned}
&T_{\sigma_z}a_1=(u+1)^{\ell^*(z)}a_{z}+\sum_{\substack{w\in \mathbf{I}_\ast\\ \rho(w)<\rho(z)}}{L_{w}^{\sigma_z}a_{w}},
\qquad L_{w}^{\sigma_z}\in u\mathbb{Z}[u],\\
&T_{\sigma_z}a_1\equiv T_{\hat{\sigma}_z}a_1\pmod{\sum_{\substack{w\in \mathbf{I}_* \\ \rho(w)<\rho(z)}}u\Z[u]a_w} .
\end{aligned}
$$
\end{cor}

\begin{cor}\label{cor:z to sigmaZ inj.} For each $z\in \mathbf{I}_\ast$ we fix a reduced $\mathbf{I}_\ast$-expression $\sigma$ of $z$ and
define $\sigma_z$ as in Definition \ref{sigmaZ}.
Then the map $\sigma_*\colon z\mapsto\sigma_z$ defines an injection from $\mathbf{I}_\ast$ into $W$. In other words,
$\sigma_{z_1}=\sigma_{z_2}$ if and only if $z_1=z_2$.
\end{cor}

\begin{proof} This follows from Proposition \ref{keyprop1}.
\end{proof}

\begin{dfn} We define $$
\A_{-1}:=\Z[u,u^{-1},(u+1)^{-1}],\,\,\,\A_{1}:=\Z[u,u^{-1},(u-1)^{-1}] .
$$
\end{dfn}

\begin{cor}\label{prop:az-Bruhat order2} For each $z\in \mathbf{I}_\ast$ we fix a reduced $\mathbf{I}_\ast$-expression $\sigma$ of $z$ and
define $\sigma_z$ as in Definition \ref{sigmaZ}. Then
\begin{align}\label{eq:az-Bruhat order2}
a_z=\frac{1}{(u+1)^{\ell^*(z)}}{T_{\sigma_z}}a_1+\sum_{w\in \mathbf{I}_\ast(\prec_{\sigma}z)}{\xi_z^{w}T_{\sigma_{w}}a_1},
\end{align}
where for each $w\in \mathbf{I}_\ast(\prec_{\sigma}z)$, $\xi_z^{w}\in\A_{-1}$. In particular, $$
\begin{aligned}\label{eq:az-Bruhat order3}
a_z&=\frac{1}{(u+1)^{\ell^*(z)}}{T_{\sigma_z}}a_1+\sum_{\substack{w\in\mathbf{I}_\ast\\ \rho(w)<\rho(z)}}{\xi_z^{w}T_{\sigma_{w}}a_1}\\
&=\frac{1}{(u+1)^{\ell^*(z)}}{T_{\sigma_z}}a_1+\sum_{\substack{w\in\mathbf{I}_\ast\\ \ell(\sigma_w)<\ell(\sigma_z)}}{\xi_z^{w}T_{\sigma_{w}}a_1}.
\end{aligned}$$
\end{cor}

\begin{proof} This follows from Proposition \ref{keyprop1}, (\ref{triangularRelation1}) and Corollary \ref{TwoReducedcor1}.
\end{proof}

Let $M_{\A,u}$ be the free $\A$-submodule of $M_u$ generated by $\{a_z \mid z\in \mathbf{I}_\ast\}$. It is clear that $M_{\A,u}$ naturally becomes a left $\HH_{\A,u}$-module. We set $$
M_{\A_{-1},u}:=\A_{-1}\otimes_{\A}M_{\A,u},\,\,\,M_{\A_{1},u}:=\A_{1}\otimes_{\A}M_{\A,u} .
$$

For each $z\in \mathbf{I}_\ast$, we identify
 $1_{\A_{-1}}\otimes_{\A} a_z$,
 $1_{\A_{1}}\otimes_{\A} a_z$,
 $1_{\A_{\pm 1}}\otimes_{\A} a_z$
 and $1_{\Q(u)}\otimes_{\A} a_z$ with $a_z$.

\begin{cor} \label{Aneg1basis} For each $z\in \mathbf{I}_\ast$ we fix a reduced $\mathbf{I}_\ast$-expression $\sigma$ of $z$ and
define $\sigma_z$ as in Definition \ref{sigmaZ}.  Then the elements in the following set \begin{equation}\label{newBasis1}
\bigl\{T_{\sigma_z}a_1\bigm|z\in \mathbf{I}_\ast\bigr\}
\end{equation}
form an $\A_{-1}$-basis of $M_{\A_{-1},u}$,  an $\A_{\pm 1}$-basis of $M_{\A_{\pm 1},u}$ and a $\Q(u)$-basis of $M_u$. The same is true if one replaces $\A_{-1}$ with any field $K$ and $u$ with any $\lam\in K^\times$ whenever there is a ring homomorphism $\phi: \A_{-1}\rightarrow K$ with $\lam=\phi(u)$.
\end{cor}

\begin{proof} This follows from Proposition \ref{keyprop1} and (\ref{triangularRelation1}).
\end{proof}

\begin{cor} \label{SurJect} For each $z\in \mathbf{I}_\ast$ we fix a reduced $\mathbf{I}_\ast$-expression $\sigma$ of $z$ and
define $\sigma_z$ as in Definition \ref{sigmaZ}. We have that $$
\HH_{\A_{-1},u}X_\emptyset=\text{$\A_{-1}$-Span}\bigl\{T_{\sigma_z}X_{\emptyset}\bigm|z\in \mathbf{I}_\ast\bigr\}.
$$
In particular, $$
\HH_{\A_{\pm 1},u}X_\emptyset=\text{$\A_{\pm 1}$-Span}\bigl\{T_{\sigma_z}X_{\emptyset}\bigm|z\in \mathbf{I}_\ast\bigr\},
$$
and the map $\mu\downarrow_{M_{\A_{\pm 1},u}}: M_{\A_{\pm 1},u}\rightarrow
\HH_{\A_{\pm 1},u}X_\emptyset$ is surjective.
\end{cor}

\begin{proof} This follows from Corollary \ref{Aneg1basis} and the surjectivity of $\mu\downarrow_{M_{\A_{-1},u}}: M_{\A_{-1},u}\twoheadrightarrow\HH_{\A_{-1},u}X_\emptyset$.
\end{proof}

For any field $K$ and any ring homomorphism $\phi: \A_{\pm 1}\rightarrow K$ with $\lam=\phi(u)$, we define $$
\mu_K: M_{\lam}\rightarrow\HH_\lam (1_K\otimes_{\A_{\pm 1}}X_\emptyset)
$$
to be the composition of the following surjection $$\Id_K\otimes_{\A_{\pm 1}}\mu\downarrow_{M_{\A_{\pm 1},u}}: M_{\lam}=K\otimes_{\A_{\pm 1}}M_{\A_{\pm 1},u}\twoheadrightarrow
K\otimes_{\A_{\pm 1}}\HH_{\A_{\pm 1},u}X_\emptyset$$ with the canonical surjective homomorphism $$\iota_K: K\otimes_{\A_{\pm 1}}\HH_{\A_{\pm 1},u}X_\emptyset\twoheadrightarrow
\HH_\lam (1_K\otimes_{\A_{\pm 1}} X_\emptyset)$$ introduced in Conjecture \ref{conj1}. By definition, we know that $\mu_K$ is surjective.

\begin{prop} \label{mainprop}  Let $K$ be a field. For any ring homomorphism $\phi: \A_{\pm 1}\rightarrow K$ with $\lam=\phi(u)$, the elements in the following set \begin{equation}\label{YZ2}
\bigl\{Y_{K,z}:=\mu_K\bigl(1_K\otimes_{\A_{\pm 1}}a_z\bigr)\in \HH_{\lam}(1\otimes_{\A_{\pm 1}} X_\emptyset)\bigm|z\in \mathbf{I}_\ast\bigr\}
\end{equation}
form an $K$-basis of $\HH_{\lam}(1_K\otimes_{\A_{\pm 1}} X_\emptyset)$. In particular, $\mu_K$ is a left $\HH_\lam$-module isomorphism. Furthermore, the elements in the following set \begin{equation}\label{YZ1}
\bigl\{Y_z:=\mu(a_z)\in \HH_{\A_{\pm 1},u}X_\emptyset\bigm|z\in \mathbf{I}_\ast\bigr\}
\end{equation}
form an $\A_{\pm 1}$-basis of $\HH_{\A_{\pm 1},u}X_\emptyset$.
\end{prop}

\begin{proof} We consider the first part of the proposition. Recall that $\mu_K$ is surjective. Since $\{1_K\otimes_{\A_{\pm 1}}a_z\mid z\in \mathbf{I}_\ast\}$ is a $K$-basis of $M_\lam$, it suffices to show that the elements in the subset (\ref{YZ2}) are $K$-linearly independent. Note that the assumption $\lam\neq 0,-1$ is used here to ensure that $\mu_K$ is surjective (by Corollary \ref{Aneg1basis} and Corollary \ref{SurJect}).

Suppose that the elements in the subset (\ref{YZ2}) are $K$-linearly dependent. That says, we can find an positive integer $m$ and $$\{z_1,z_2,\cdots,z_m\}\subseteq \mathbf{I}_\ast, $$ such that $Y_{K,z_1},\cdots, Y_{K,z_m}$ are $K$-linearly dependent. For each $z\in\{z_1,\cdots,z_m\}$, we fix a reduced $\mathbf{I}_\ast$-expression of $z$ and define $\sigma_z$ as in Definition \ref{sigmaZ}.
Without loss of generality, we can assume that  $\rho(z_1)\leq \rho(z_2) \leq \cdots \leq \rho(z_m)$. Equivalently,
$\ell(\sigma_{z_1}) \leq \ell(\sigma_{z_2})\leq \cdots \leq \ell(\sigma_{z_m})$. Furthermore, we can find an integer $n\geq m$ and  a finite subset $\{w_{m+1},\cdots,w_{n}\}$ of $W\setminus\{\sigma_{z_1},\cdots,\sigma_{z_m}\}$ such that for each $1\leq j\leq m$, \begin{equation}\label{TriangularRelations2}
Y_{z_j}=\mu(a_{z_j})=\sum_{i=1}^{m}\widetilde{L}^{\sigma_{z_i}}_{z_j}T_{\sigma_{z_i}}+
\sum_{k=m+1}^{n}\widetilde{L}^{w_k}_{z_j}T_{w_k} .
\end{equation}
By Definition \ref{Stru1}, Definition \ref{Stru2}, Lemma \ref{keylem1} and Proposition \ref{keyprop1}, $$
(Y_{z_1},Y_{z_2},\cdots,Y_{z_m})=(\underbrace{T_{\sigma_{z_1}},T_{\sigma_{z_2}},\cdots,T_{\sigma_{z_m}}}_{\text{$m$ terms}},\underbrace{T_{w_{m+1}},T_{w_{m+2}},\cdots,T_{w_n}}_{\text{$(n-m)$ terms}})D_1A_uD_2,
$$
where $D_1$ is the following $n\times n$ diagonal matrix: $$
D_1=\diag\bigl((-1)^{\ell(\sigma_{z_1})},\cdots,(-1)^{\ell(\sigma_{z_m})},(-1)^{\ell(w_{m+1})},\cdots,(-1)^{\ell(w_{n})}\bigr),
$$
$D_2$ is the following $n\times n$ diagonal matrix: $$
D_2=\diag\bigl(\epsilon(z_1),\cdots,\epsilon(z_m)\bigr) ,
$$
and $A_u$ is the following $n\times m$ matrix in $M_{n\times m}(\Z[u^{-1}])$:
\begin{align*}
A_u=&
  \left(
\renewcommand\arraystretch{1.8}
  \begin{array}{cccc}
  (1-u^{-1})^{\ell^*(z_1)} &    &  & \\
   & (1-u^{-1})^{\ell^*(z_2)}
  &  \multicolumn{2}{c}{\raisebox{1.3ex}[0pt]{\Huge $0$}}  \\
   &   & \ddots &   \\
\multicolumn{2}{c}{\raisebox{1.3ex}[0pt]{\Huge $*$}}
&
& (1-u^{-1})^{\ell^*(z_m)} \\
{\raisebox{0ex}[0pt]{\Huge $*$}} & {\raisebox{0ex}[0pt]{\Huge $*$}} & {\raisebox{0ex}[0pt]{\Huge $*$}}  &  {\raisebox{0ex}[0pt]{\Huge $*$}} \\
  \end{array}
  \right)_{n\times m},
\end{align*}
such that the top $m\times m$ submatrix is a lower triangular matrix with diagonal elements given by $$
\bigl\{(1-u^{-1})^{\ell^*(z_1)},\cdots,(1-u^{-1})^{\ell^*(z_m)}\bigr\}.
$$

By assumption $\lam=\phi(u)\not\in\{0,1,-1\}$. We define $A_\lam:=A_u\downarrow_{u:=\lam}$. Then $$\begin{aligned}
&(1_K\otimes_{\A_{\pm 1}}Y_{z_1},\cdots,1_K\otimes_{\A_{\pm 1}}Y_{z_m})=\\
& (\underbrace{1_K\otimes_{\A_{\pm 1}}T_{\sigma_{z_1}},\cdots,1_K\otimes_{\A_{\pm 1}}T_{\sigma_{z_m}}}_{\text{$m$ terms}},\underbrace{1_K\otimes_{\A_{\pm 1}}T_{w_{m+1}},\cdots,1_K\otimes_{\A_{\pm 1}}T_{w_n}}_{\text{$(n-m)$ terms}})D_1A_\lam D_2,
\end{aligned}
$$
By the above discussion and the assumption that $\lam\neq 1$ we can see that $\rank A_\lam=m$. Note that $\{1_K\otimes_{\A_{\pm 1}}T_w \mid w\in W\}$ is a subset of
$K$-linearly independent elements in $\hat{\HH}_\lam$. Since $D_1,D_2$ are invertible, it follows that $$
\{Y_{K,z_1}=1_K\otimes_{\A_{\pm 1}}Y_{z_1},\cdots,Y_{K,z_m}=1_K\otimes_{\A_{\pm 1}}Y_{z_m}\}
$$ is a set of $K$-linearly independent elements in $\HH_\lam(1_K\otimes_{\A_{\pm 1}} X_\emptyset)\subseteq \hat{\HH}_\lam$. We get a contradiction. In particular, this implies that $\mu_K$ is a left $\HH_\lam$-module isomorphism. This proves the first part of the proposition.

Finally, taking $K=\Q(u)$ we see that $\mu_{\Q(u)}$ is an isomorphism by the first part of the proposition which we have just proved. This further implies that $\Id_{\Q(u)}\otimes_{\A_{\pm 1}}\mu\downarrow_{M_{\A_{\pm 1},u}}$ is an isomorphism. Since $$
\Q(u)\otimes_{\A_{\pm 1}}\Ker\mu\downarrow_{M_{\A_{\pm 1},u}}\subseteq \Ker(\Id_{\Q(u)}\otimes_{\A_{\pm 1}}\mu\downarrow_{M_{\A_{\pm 1},u}})=\{0\} ,
$$
it follows that $\Ker\mu\downarrow_{M_{\A_{\pm 1},u}}=0$. Hence $\mu\downarrow_{M_{\A_{\pm 1},u}}$ is an isomorphism and the elements in (\ref{YZ1}) form an $\A_{\pm 1}$-basis of $\HH_{\A_{\pm 1},u}X_\emptyset$.
This proves the second part of the proposition and hence we complete the proof of the proposition.
\end{proof}

\medskip

\noindent
{\bf Proof of Conjecture \ref{conj1}:} (1) and (2) follows from Proposition \ref{mainprop}. Now (3) follows from (1) and (2). It remains to consider (4).
For this purpose, we assume that $W$ is finite. Then $X_\emptyset\in\HH_{\A_{\pm 1},u}$.

By (\ref{TriangularRelations2}) and Proposition \ref{keyprop1}, we easily see that the elements in the following set $$
\{Y_{z} \mid z\in \mathbf{I}_\ast\}\sqcup\{T_{w} \mid w\in W\setminus\{\sigma_z \mid z\in \mathbf{I}_\ast\}\}
$$
form an $\A_{\pm 1}$-basis of $\hat{\HH}_{\A_{\pm 1},u}={\HH}_{\A_{\pm 1},u}$. This implies that $\HH_{\A_{\pm 1},u}X_{\emptyset}$ is a pure and free $\A_{\pm 1}$-submodule of ${\HH}_{\A_{\pm 1},u}$. This completes the proof of Conjecture \ref{conj1}.
\bigskip

\begin{cor} \label{newBasis2}
The elements in the following set \begin{equation}\label{newBasis22}
\bigl\{T_{\sigma_z}X_\emptyset\bigm|z\in \mathbf{I}_\ast\bigr\}
\end{equation}
form an $\A_{\pm 1}$-basis of ${\HH}_{\A_{\pm 1},u} X_\emptyset$. The same is true if one replaces $\A_{\pm 1}$ with any field $K$ and $u$ with any $\lam\in K^\times$ whenever there is a ring homomorphism $\phi: \A_{\pm 1}\rightarrow K$ with $\lam=\phi(u)$.
\end{cor}

\begin{proof} This follows from Corollary \ref{Aneg1basis}, Corollary \ref{SurJect} and Conjecture \ref{conj1} (which we have just proved).
\end{proof}

By Lemma \ref{keylem1}, $$
\mu(a_z)=\sum_{x\in W}\widetilde{L}_z^x T_x , $$
where $\widetilde{L}_z^x\in\Z[u^{-1}]$. Following \cite[Theorem 0.2(b)]{Lu3}, we define $n_z^x:=\widetilde{L}_z^x\downarrow_{u^{-1}=0}\in\Z$.
Then Lusztig has proved in \cite[Theorem 0.2(c)]{Lu3} that there is a unique surjective function $\pi: W\twoheadrightarrow \mathbf{I}_\ast$ such that for any $x\in W, z\in\mathbf{I}_\ast$, we have $n_z^x=1$ if $z=\pi(x)$; and $n_z^x=0$ if $z\neq\pi(x)$.

Our next result shows that the map $\sigma_*$ which we introduced in Corollary \ref{cor:z to sigmaZ inj.} is a right inverse of $\pi$.

\begin{cor} Let $\sigma_*\colon \mathbf{I}_\ast\hookrightarrow W$ be the injection defined in Corollary \ref{cor:z to sigmaZ inj.}. Then
$\pi\circ\sigma_*=\id_{\mathbf{I}_*}$.
\end{cor}

\begin{proof} Let $z\in\mathbf{I}_\ast$. Following \cite[\S1.8]{Lu3}, we use $\{\underline{T}_w \mid w\in W\}$ to denote the standard basis of the specialization $\HH_0$ of $\HH_u$ at $u:=0$, and use $M_0$ to denote the specialization of $M$ at $u:=0$. Then $M_0$ is a $\Q$-space with basis $\{\underline{a}_w \mid w\in\mathbf{I}_\ast\}$ and with $\HH_0$-module structure given by $$\begin{aligned}
& \underline{T}_s\underline{a}_w=\underline{a}_{sw}\,\,\,\text{if $sw=ws^\ast>w$;}\\
& \underline{T}_s\underline{a}_w=\underline{a}_{sws^\ast}\,\,\,\text{if $sw\neq ws^*>w$;}\\
& \underline{T}_s\underline{a}_w=-\underline{a}_{w}\,\,\,\text{if $sw<w$,}\end{aligned}
$$
where $s\in S, w\in\mathbf{I}_\ast$.

Setting $u:=0$ on both sides of (\ref{triangularRelation1}), we get that  $$
\underline{T}_{\sigma_z}\underline{a}_1=\underline{a}_{z} .
$$
On the other hand, since $\widetilde{L}^{\sigma_z}_x:=(-1)^{\ell(\sigma_z)}\eps(x)\overline{L^{\sigma_z}_x}$ for any $x\in\mathbf{I}_\ast$,
setting $u:=0$ in $L^{\sigma_z}_x$ is equivalent to setting $u^{-1}:=0$ in $\widetilde{L}^{\sigma_z}_x$. We can deduce from \cite[Theorem 0.2(c)]{Lu3}
that $$
\pi\circ\sigma_*(z)=\pi(\sigma_z)=z.
$$
Note that $(-1)^{\ell(\sigma_z)}\eps(z)=(-1)^{\rho(z)}(-1)^{\rho(z)}=1$. This completes the proof of the corollary.
\end{proof}

\section*{Acknowledgements}

The research was carried out under the support from the  National Natural Science Foundation of China (No. 11525102, 11471315).

\end{document}